\documentclass[a4paper,11pt]{article}
\usepackage[utf8]{inputenc}
\usepackage[T1]{fontenc}

\usepackage{indentfirst}
\usepackage{enumitem}
\usepackage{amsmath}
\usepackage{amssymb,amsmath,amsthm}
\usepackage{ebproof}
\usepackage{aliascnt}
\usepackage{hyperref}
\usepackage[capitalise]{cleveref}
\usepackage{ragged2e}
\usepackage{adjustbox}
\usepackage [autostyle, english = american]{csquotes}
\MakeOuterQuote{"}

\usepackage[backend=biber]{biblatex}
\usepackage[affil-it]{authblk}

\usepackage{tikz-cd}
\usetikzlibrary{decorations.markings,intersections}
\tikzset{shorten <>/.style={shorten >=#1,shorten <=#1}}
\tikzset{every picture/.prefix code=\DisableQuotes}

\tikzset{%
scalearrow/.style n args={3}{
  decoration={
    markings,
    mark=at position (1-#1)/2*\pgfdecoratedpathlength
      with {\coordinate (#2);},
    mark=at position (1+#1)/2*\pgfdecoratedpathlength
      with {\coordinate (#3);},
    },
  postaction=decorate,
  }
}

\usepackage{quiver}
\usepackage[mathcal]{euscript}
\usepackage{ stmaryrd }
\usepackage{ dsfont }
\usepackage{upgreek}
\usepackage{ mathrsfs }
\usepackage{comment}

 \let\oldtheorem\newtheorem
\RenewDocumentCommand{\newtheorem}{s m o m O{}}{%
\IfBooleanTF{#1}%
{\oldtheorem{#2}{#4}}%
{\IfNoValueTF{#3}{\oldtheorem{#2}{#4}[#5]}%
{\newaliascnt{#2}{#3}%
\oldtheorem{#2}[#2]{#4}%
\aliascntresetthe{#2}}}}

\newtheorem{theorem}{Theorem}[section]
\newtheorem{proposition}[theorem]{Proposition}
\newtheorem{lemma}[theorem]{Lemma}
\newtheorem{corollary}[theorem]{Corollary}

\theoremstyle{definition}
\newtheorem{definition}{Definition}[section]

\theoremstyle{remark}
\newtheorem{remark}{Remark}[section]

\title{A notion of semi-cubical tribe}
\author{El Mehdi Cherradi}
\affil{IRIF - CNRS - Universit\'e Paris Cit\'e \\MINES ParisTech - Universit\'e PSL}
\date{}

\begin{document}

\maketitle

\begin{abstract}
The important notions our ideas revolve around are that of tribes, a class of categories aiming at modeling intensional type theory, and that of semi-cubical objects, for a category of cubes with symmetries and reversals. 
We introduce a general notion of $J$-tribes, tribes suitably enriched in presheaves over $J$, and $J$-frames in a tribe $\mathcal{T}$ as a category of $J$-shaped resolutions in $\mathcal{T}$, for a direct category $J$, and we construct the $J$-tribe of $J$-frames in $\mathcal{T}$. Similar notions of $R$-tribes and $R$-frames are introduced for $R$ a suitable generalized direct category, and corresponding properties are established thanks to a construction relating $R$ to a strictly direct category. In particular, our framework applies to semi-cubes, which is our main motivation, and yields an appropriate notion of a semi-cubical tribe.
\end{abstract}

\tableofcontents

\newpage 

\section*{Introduction}

\subsection*{Models of type theory, tribes and inverse diagrams}
Introduced by Joyal in \cite{joyal2017notes}, the notion of tribe provides a categorical framework for the semantics of intensional type theory: a tribe is essentially a category $\mathcal{T}$ equipped with a notion of fibrations, thought of as display maps, or context projections, which are stable under base change, and whose corresponding model of type theory supports identity types. Given such a tribe $\mathcal{T}$, one obtains a canonical comprehension category given by the codomain functor $$\mathbf{cod} :  \mathcal{T}^{\rightarrow_{\mathbf{fib}}} \to \mathcal{T},$$ where  $\mathcal{T}^{\rightarrow_{\mathbf{fib}}}$ is the full subcategory of the arrow category spanned by the fibrations.

Given a categorical model of type theory $\mathbf{T}$ in a broad sense, and a small category $I$ assumed to be inverse, the category of diagrams $\mathbf{T}^I$ is known to inherit enough features of $\mathbf{T}$ to also provide a model of a similar type theory, as studied in \cite{kapulkin2021homotopical} in the context of categories with attributes.
An interesting inverse category to consider is the opposite of the semi-simplex category $\Delta_\sharp$, whose objects are the natural numbers and whose maps are the injective monotone functions between them. Type theories that can be thought of as arising from (semi)-simplicial/cubical diagrams in a given base model, or at least having inbuilt dependencies on primitive shapes (simplices or cubes) have been considered recently in the literature as in \cite{cohen2016cubical} (cubical type theory) and \cite{riehl2017type} (simplicial type theory). One crucial feature of such "diagrammatic" models, for instance in the semi-simplicial case, is that they inherit a canonical enrichment in finite semi-simplicial sets: in particular, for every type $A$, one can form the type $A^{\Delta^1}$, where $\Delta^1$ plays the role of an interval object.

\subsection*{Path objects and cubical diagrams}

Another interesting direction is that of the category of semi-cubes, that is, the free monoidal category $\square_\sharp$ generated by two face maps $\delta_i : I^0 \to I^1$ with domain the monoidal unit. Its objects can be thought of as finite powers of an interval $I$ (cubes), and its maps are inclusions of sub-cubes. A key observation is that the mere existence of identity types already translates nearly into a cubical enrichment. Indeed, given a fibration category $\mathcal{F}$ and an object $x$ of $\mathcal{F}$, one may consider a path object for $x$, that is a factorization of the diagonal $x \to x \times x$ as a weak equivalence $x \to Px$ followed by a fibration $Px \to x \times x$. Iterating this construction, the family of path objects $P^n x$ defines a semi-cubical object in $\mathcal{F}$:
$$P^\bullet x :  \square_\sharp^{op} \to \mathcal{F}$$
Note that, in particular, we also get a semi-simplicial object in $\mathcal{F}$. 

However, this does not define a functor $\mathcal{F} \to \mathcal{F}^{\square_\sharp^{op}}$ (nor a functor $\mathcal{F} \to \mathcal{F}^{\Delta_\sharp^{op}}$) since the path objects are not functorial here.  A natural direction is hence to make these diagrams canonical, by introducing definitions of fibration categories/tribes enriched in a suitable way over semi-simplicial sets, and it is reasonable to expect the corresponding categories to be DK-equivalent to their unenriched counterparts.

\subsection*{Semi-simplicial tribes}

In Section 3 of \cite{ks2019internal}, the authors investigate such a notion of semi-simplicial tribe, based on the work of \cite{schwede2013p} in the setting of (co)fibration categories. Their aim is to provide a "replacement" of the category of tribes $\mathbf{Trb}$ that canonically enjoys the structure of a fibration category, unlike the whole category of tribes, which does not have path-objects in general (at least not canonically). The key feature enjoyed by semi-simplicial tribes is the existence of a canonical factorization of the diagonals $\mathcal{T} \to \mathcal{T} \times \mathcal{T}$ through tribes $P \mathcal{T}$ of Reedy fibrant homotopical spans in $\mathcal{T}$, defined by considering the mapping $x \mapsto x^{\Delta^1}$ (where $x^{\Delta^1}$ comes with two projections to $x$).

\subsection*{Semi-cubes with symmetries and reversals}

In the case of a tribe $\mathcal{T}$,  using the lifting property of the anodyne reflexivity maps $P^{n-1}x \to P^nx$ against the fibrations $P^n x \to P^{n-1}x  \times P^{n-1}x$, as in the diagram below,
\[\begin{tikzcd}[ampersand replacement=\&]
	{P^{n-1} x} \&\& {P^n x} \\
	\\
	{P^n x} \&\& {P^{n-1} x \times P^{n-1} x}
	\arrow["\sim"{description}, from=1-1, to=1-3]
	\arrow["\sim"{description}, from=1-1, to=3-1]
	\arrow["{<\pi_0,\pi_1>}", two heads, from=1-3, to=3-3]
	\arrow["h"{description}, dashed, from=3-1, to=1-3]
	\arrow["{<\pi_1,\pi_0>}"', two heads, from=3-1, to=3-3]
\end{tikzcd}\]
we can "lift" these semi-cubical objects, at the level of the homotopy category $\mathbf{Ho}(\mathcal{T})$, to diagrams over the opposite of a category of semi-cubes $\square_\sharp^s$ that also has symmetries (derived from an automorphism $\tau : I^1 \otimes I^1 \simeq I^1 \otimes I^1$ "permuting" the intervals in the monoidal product) and reversals (obtained from $r : I^1 \simeq I^1$ "reversing" the direction):
\[\begin{tikzcd}[ampersand replacement=\&]
	{\square_\sharp^{op}} \&\& {(\square_\sharp^s)^{op}} \\
	\\
	{\mathcal{T}} \&\& {\mathbf{Ho}(\mathcal{T})}
	\arrow[from=1-1, to=1-3]
	\arrow["{P^\bullet x}"', from=1-1, to=3-1]
	\arrow["{P^\bullet_s x}", from=1-3, to=3-3]
	\arrow[from=3-1, to=3-3]
\end{tikzcd}\]

\subsection*{Application}

The notion of semi-simplicial tribe is a key ingredient in the proof strategy of \cite{ks2019internal} to establish the DK-equivalence between the relative categories of tribes $\mathbf{Trb}$ and that of finitely complete quasicategories, $\mathbf{QCat}_{lex}$. The key point is that the relative category $\mathbf{sTrb}$ of semi-simplicial tribes enjoys a canonical fibration category structure and is directly provable to be DK-equivalent to $\mathbf{Trb}$, so one can use the former as a replacement for the latter in the setting of fibration categories (whose theory considerably simplifies the study of DK-equivalences).  However, the proof pattern does not adapt directly to $\pi$-tribes because the subcategory  $\mathbf{sTrb}_\pi$ of semi-simplicial tribes which are $\pi$-tribes does not canonically enjoy a fibration category structure anymore. An important idea is that cubes are better behaved than simplices with respect to $\Pi$-types: this has been the motivating idea for this work, and a key ingredient in the proof that the relative category $\mathbb{Trb}_\pi$ of $\pi$-tribes is DK-equivalent to that of locally cartesian closed quasicategories, $\mathbf{QCat}_{lcc}$ (the so-called internal language conjecture for locally cartesian closed  $(\infty,1)$-categories).

\section*{Outline}

The goal of this document is to adapt the work of Kapulkin and Szumiło \cite{ks2019internal} to semi-cubes. To do so, we give general definitions for $J$-tribes and $J$-frames in a tribe in the first section, for $J$ a direct category. We establish the main properties of $J$-frames in Section 2, yielding \cref{J_frames_th}. In Section 3, we recall a construction for generalized inverse diagrams in a tribe, investigated in \cite{cherradi2026generalized}. This allows us to introduce in Section 4 the notion of $R$-frames for $R$ a generalized direct category, and derive the expected properties. Section 5 then tackles the problem of the enrichment of $R$-tribes. In Section 6, we prove a DK-equivalence between a tribe $\mathcal{T}$ and the tribe of $R$-frames in $\mathcal{T}$. Section 7 establishes a suitable enrichment for $R$-frames, so that tribes of $R$-frames are $R$-tribes. Finally, Section 8 wraps up the article by applying the main results to the case where $R$ is the category $\square_\sharp^s$ of semi-cubes with symmetries and reversals.

\subsection*{Acknowledgments}

The author would like to thank the anonymous referees for detailed feedback on a draft of this document.

\section{ \texorpdfstring{$J$}{\textbf{J}}-tribes and \texorpdfstring{$J$}{\textbf{J}}-frames in a tribe}  

In this section, we adapt the definition of semi-simplicial tribe to work in a more general setting, notably when $J$ is a direct category endowed with a monoidal structure (and we write $\otimes :  J \times J \to J$ for the monoidal product). 
Hence, the category $\mathbf{Set}^{J^{op}}$ of presheaves on $J$ also comes with a monoidal structure: it is given by defining the geometric product of two presheaves $K$ and $L$ on $J$ from the coend formula below, which is just the Day convolution arising from the monoidal structure on $J$:
 
 $$K \otimes L := \int^{i,j \in J} K_i \times L_j \times J^{i\otimes j}$$

While the category of semi-cubes we consider in this article is not direct, we will see that it is possible to reduce the problem to one fitting this setting. 
We introduce a general notion of a $J$-tribe, in the spirit of \cite[Definitions 3.1 and 3.2]{ks2019internal}.
  
 \begin{definition}
 \label[definition]{semicubical_tribe}
 Consider a small category $J$. A $J$-tribe $\mathcal{T}$ is a tribe enriched over presheaves on $J$, with the monoidal structure given by the geometric product above, and admitting cotensors by finite presheaves (those presheaves that are canonically finite colimits of representable ones) such that the following properties hold:
  \begin{itemize}
   \item If $i : K \to L$ is a monomorphism between finite presheaves, and if $p : a \to b$ is a fibration in $\mathcal{T}$, then the gap map $$a^L \to a^K \times_{b^K} b^L$$
   is a fibration, and is moreover a trivial fibration whenever $p$ is a trivial fibration.
   
   \item If $f :  i \to j$ is a morphism in $J$, and $a$ is an object in $\mathcal{T}$, then $a^{Hom(-,f)}$ is a weak equivalence in $\mathcal{T}$ (where $Hom(-,f) : Hom(-,i) \to Hom(-,j)$ is the representable natural transformation represented by $f$).
   
   \item If $K$ is a finite presheaf on $J$, and if $p : a \to b$ is anodyne, then $p^K : a^K \to b^K$ is also anodyne. 
  \end{itemize}
 \end{definition}

 A convenient class of categories we wish to work with is that of Reedy categories. A good reference for the theory of such categories, that we will rely on, is the account of Riehl and Verity \cite{riehlreedy}.
 
 We recall the definition of elegant Reedy categories, introduced in \cite{bergner2013reedy}, as well as an important result:
 
 \begin{definition}
  A Reedy category $J$ is an elegant Reedy category when for every monomorphism $m : X \to Y$ in $\mathbf{Set}^{J^{op}}$ and every object $j$ of $J$, the relative latching map $L_j m$ is a monomorphism.
 \end{definition}
 
 Examples of elegant Reedy categories include the category $\Delta$ as well as all direct categories. In what follows, we fix an elegant Reedy category $J$.
 
 \begin{lemma}
 \label[lemma]{mono_reedy}
 The category of presheaves $\mathbf{Set}^{J^{op}}$ admits a cofibrantly generated weak factorization system whose left class is the class of monomorphisms. 
  
  Moreover, this class is generated by the border inclusions $$\partial J^x \to J^x$$
  where $J^x$ is the representable presheaf represented by $x$, and $\partial J^x$ is its subpresheaf given as the "latching" colimit (i.e., $\partial J^x \to J^x$ is the latching map at $x$ of the Yoneda embedding $J \to \mathbf{Set}^{J^{op}}$).
 \end{lemma}
 
 \begin{proof}
  This is Corollary 6.8 (and Example 6.9) in \cite{riehlreedy}, by definition of elegant Reedy categories.
 \end{proof}

 \begin{lemma}
 \label[lemma]{skeletal_filtration}
  Any presheaf $K \in \mathbf{Set}^{J^{op}}$ can be decomposed as a (possibly transfinite) composite of inclusions
  $$\mathbf{sk}^0 K \to ... \to \mathbf{sk}^n K ...$$
  where the successive maps are obtained by pushouts:

\[\begin{tikzcd}[ampersand replacement=\&]
	{\amalg_{S_n}\partial J^x} \&\& {\amalg_{S_n} J^x} \\
	\\
	{\mathbf{sk}_{n-1} K} \&\& {\mathbf{sk}_{n} K}
	\arrow[from=1-1, to=1-3]
	\arrow[from=1-1, to=3-1]
	\arrow[from=1-3, to=3-3]
	\arrow[from=3-1, to=3-3]
	\arrow["\ulcorner"{anchor=center, pos=0.125, rotate=180}, draw=none, from=3-3, to=1-1]
\end{tikzcd}\]
  for $S_n$ the set of maps $J^x \to K$ with $\mathbf{deg} x = n$.
 \end{lemma}
 
\begin{proof}
  This is an instance of Proposition 6.3 of \cite{riehlreedy}.
\end{proof}
 
 We now suppose that $J$ is a direct category, so it makes sense to talk about fibrant $J$-diagrams in a tribe.
 
 \begin{definition}
  We define a $J$-frame in a tribe $\mathcal{T}$ to be a Reedy fibrant homotopical diagram
  $$J^{op} \to \mathcal{T}$$
  where the direct category $J$ is given the homotopical structure with all maps being weak equivalences. We write $J\text{-Fr}(\mathcal{T})$ for the category of frames in $\mathcal{T}$.
 \end{definition}
 
   Consider a frame $F : J^{op} \to \mathcal{T}$ and an object $z$ of $\mathcal{T}$. We write  $\mathcal{T}(z,F)$ for the presheaf defined as the composite:

\[\begin{tikzcd}[ampersand replacement=\&]
	{J^{op}} \&\& {\mathcal{T}} \&\& {\mathbf{Set}}
	\arrow["F", from=1-1, to=1-3]
	\arrow["{Hom(z,-)}", from=1-3, to=1-5]
\end{tikzcd}\]

\section{Properties of \texorpdfstring{$J$}{\textbf{J}}-frames}

The theory of $J$-frames is very similar to that of semi-simplicial frames, developed in  \cite{schwede2013p}, and which guides us here.

For $K$ a finite presheaf on $J$, we write $F^K$ for a representing object of the functor $\mathcal{T}^{op} \to \mathbf{Set}$ mapping $z$ to $Hom_{\mathbf{Set}^{J^{op}}}(K,\mathcal{T}(z,F))$, provided such an object exists (this is an unenriched weighted limit).

The following lemma is an analogue of Proposition 3.3 in \cite{schwede2013p}, which essentially relies on the existence of skeletal filtrations (\cref{skeletal_filtration}) to inductively build up the objects $F^K$.

\begin{lemma}
\label[lemma]{cotensor_reedy_cubes}
 Consider a tribe $\mathcal{T}$ and a $J$-frame $F$ in $\mathcal{T}$.
 For any finite presheaf $K \in  \mathbf{Set}^{J^{op}}$, $F^K$ exists in $\mathcal{T}$. Moreover, the functor $F \mapsto F^K$ takes any (Reedy) fibration (resp. trivial fibration) between frames to a fibration (resp. a trivial fibration), and it takes anodyne maps to anodyne maps.
\end{lemma}

\begin{proof}
 The existence of $F^K$ is tautological when $K$ is the representable $J^x$ for $x \in J$ (the representing object is $F(x)$). This is also true for the boundaries $\partial J^x$, as, by definition of Reedy fibrancy, the matching object $M_x F$ exists. In this case, the stated property of the functor $F \mapsto F^K$ follows from a general result: in a fibration category $\mathcal{F}$, a Reedy fibration $X \to Y$ in $\mathcal{F}^{J^{op}}$ induces a pointwise fibration $X_c \to Y_c$ and a fibration between the matching objects $M_cX \to M_cY$ at any object $c \in J$, as proved in \cite[Corollary 9.3.6 (2)]{radulescu2006cofibrations}.

The general case follows by induction on the dimension of $K$, taking advantage of the existence of a skeletal filtration for $K$ (\cref{skeletal_filtration}).
Explicitly, when $K$ is of dimension $0$ (i.e., when it has only cells of degree $0$), the result is trivial: the representing object is the terminal one. Assuming that the result holds for all dimensions up to $n-1$, we can form the following pullback square,
\[\begin{tikzcd}[ampersand replacement=\&]
	P \&\& {\Pi_{S_n}F^{J^x}} \\
	\\
	{F^{\mathbf{sk}_{n-1} K}} \&\& {\Pi_{S_n}F^{\partial J^x}}
	\arrow[from=1-1, to=1-3]
	\arrow[from=1-1, to=3-1]
	\arrow["\ulcorner"{anchor=center, pos=0.125}, draw=none, from=1-1, to=3-3]
	\arrow[from=1-3, to=3-3]
	\arrow[from=3-1, to=3-3]
\end{tikzcd}\]
where $S_n$ is the set of cells $J^x \to K$ of degree $n$, since the vertical map on the right is a fibration by the Reedy fibrancy assumption on $F$. The object $P$ is then the representing object $F^K$ we were looking for.

If $p : F \to F'$ is a fibration of $J$-frames, the induced map between the pullbacks in the following diagram
\[\begin{tikzcd}[ampersand replacement=\&]
	{F^K} \&\& {\Pi_{S_n}F^{J^x}} \\
	\& {F'^K} \&\& {\Pi_{S_n}F'^{ J^x}} \\
	{F^{\mathbf{sk}_{n-1} K}} \&\& {\Pi_{S_n}F^{\partial J^x}} \\
	\& {F'^{\mathbf{sk}_{n-1} K}} \&\& {\Pi_{S_n}F'^{\partial J^x}}
	\arrow[from=1-1, to=1-3]
	\arrow[dashed, from=1-1, to=2-2]
	\arrow[from=1-1, to=3-1]
	\arrow["\ulcorner"{anchor=center, pos=0.125}, draw=none, from=1-1, to=3-3]
	\arrow[from=1-3, to=2-4]
	\arrow[from=1-3, to=3-3]
	\arrow[from=2-2, to=2-4]
	\arrow[from=2-2, to=4-2]
	\arrow["\ulcorner"{anchor=center, pos=0.125}, draw=none, from=2-2, to=4-4]
	\arrow[from=2-4, to=4-4]
	\arrow[from=3-1, to=3-3]
	\arrow[from=3-1, to=4-2]
	\arrow[from=3-3, to=4-4]
	\arrow[from=4-2, to=4-4]
\end{tikzcd}\]
factors as $$F^K \simeq \Pi_{S_n}F^{J^x} \times_{\Pi_{S_n}F^{\partial J^x}} F^{\mathbf{sk}_{n-1} K} \to \Pi_{S_n}F'^{J^x} \times_{\Pi_{S_n}F'^{\partial J^x}} F^{\mathbf{sk}_{n-1} K}$$

followed by $$\Pi_{S_n}F'^{J^x} \times_{\Pi_{S_n}F'^{\partial J^x}} F^{\mathbf{sk}_{n-1} K} \to \Pi_{S_n}F'^{J^x} \times_{\Pi_{S_n}F'^{\partial J^x}} F'^{\mathbf{sk}_{n-1} K} \simeq F'^K $$

where the first map is a fibration as a base change of a finite product of the fibration $F^{J^x} \to F'^{J^x} \times_{F'^{\partial J^x}} F^{\partial J^x}$ (by the assumption that $p$ is a Reedy fibration), and the second map is a fibration, being a base change of the fibration $p^{\mathbf{sk}_{n-1} K}$ (by our inductive hypothesis). Observe that these two maps are also weak equivalences whenever $p$ is one (by induction or by the characterization of Reedy trivial fibration). The argument applies to anodyne maps $w :  F \to F'$ since pullbacks of anodyne maps along fibrations are anodyne in a tribe.
\end{proof}

In this context, we can derive a pullback-cotensor property:

 \begin{lemma}
 \label[lemma]{cotensor_reedy_cubes_gap}
  For any monomorphism $i : K \to L$ between finite presheaves on $J$, and any fibration $p : F \to F'$ between $J$-frames, the gap map in the diagram below

\[\begin{tikzcd}[ampersand replacement=\&]
	{F^L} \\
	\& {F'^L \times_{F'^K} F^K} \&\& {F^K} \\
	\\
	\& {F'^L} \&\& {F'^K}
	\arrow[dashed, from=1-1, to=2-2]
	\arrow[curve={height=-12pt}, from=1-1, to=2-4]
	\arrow[curve={height=12pt}, from=1-1, to=4-2]
	\arrow[from=2-2, to=2-4]
	\arrow[from=2-2, to=4-2]
	\arrow["\ulcorner"{anchor=center, pos=0.125}, draw=none, from=2-2, to=4-4]
	\arrow[from=2-4, to=4-4]
	\arrow[from=4-2, to=4-4]
\end{tikzcd}\]
  is a fibration in $\mathcal{T}$, that is moreover a trivial fibration whenever $p$ is so.
 \end{lemma}
 
 \begin{proof}
  This is proved following the pattern of the proof for Proposition 3.5 of \cite{schwede2013p}. First note that $F^K \to F'^K$ is a fibration by \cref{cotensor_reedy_cubes}, so the pullback indeed exists in the tribe $\mathcal{T}$. We proceed by induction on the number of cells in $L$ that are not in the image of $i: K \to L$. If this number is $0$, it means that the monomorphism $i$ is surjective in every dimension, so that $i$ is in fact an isomorphism. Then, the pullback is isomorphic to $F^L$, and the gap map is an isomorphism (hence a fibration). If the number is strictly positive, we can choose a factorization $j \circ i' : K \to L' \to L$ of $i$ such that $K \to L'$ is a monomorphism and $L'$ has one cell fewer than $L$ (say, $J^x \to L$). We observe that the gap map $$F^L \to F'^L \times_{F'^K} F^K$$ factors as shown in the diagram below,

  \adjustbox{scale=0.8}{
  \begin{tikzcd}[ampersand replacement=\&]
	{F^{J^x}} \&\& {F^L} \\
	{F'^{J^x} \times_{F'^{\partial J^x}} F^{\partial J^x}} \&\& {F'^L \times_{F'^{L'}} F^{L'}} \&\& {F^{L'}} \\
	\&\& {F'^L \times_{F'^K} F^K} \&\& {F'^{L'} \times_{F'^K} F^K} \&\& {F^K} \\
	\\
	\&\& {F'^L} \&\& {F'^{L'}} \&\& {F'^K}
	\arrow[two heads, from=1-1, to=2-1]
	\arrow[from=1-3, to=1-1]
	\arrow["\ulcorner"{anchor=center, pos=0.125, rotate=-135}, draw=none, from=1-3, to=2-1]
	\arrow[two heads, from=1-3, to=2-3]
	\arrow[from=1-3, to=2-5]
	\arrow[from=2-3, to=2-1]
	\arrow[from=2-3, to=2-5]
	\arrow[two heads, from=2-3, to=3-3]
	\arrow["\ulcorner"{anchor=center, pos=0.125, rotate=45}, draw=none, from=2-3, to=3-5]
	\arrow[two heads, from=2-5, to=3-5]
	\arrow[from=2-5, to=3-7]
	\arrow[from=3-3, to=3-5]
	\arrow[two heads, from=3-3, to=5-3]
	\arrow["\ulcorner"{anchor=center, pos=0.125}, draw=none, from=3-3, to=5-5]
	\arrow[from=3-5, to=3-7]
	\arrow[two heads, from=3-5, to=5-5]
	\arrow["\ulcorner"{anchor=center, pos=0.125}, draw=none, from=3-5, to=5-7]
	\arrow[two heads, from=3-7, to=5-7]
	\arrow[from=5-3, to=5-5]
	\arrow[from=5-5, to=5-7]
\end{tikzcd}
}

where $$F^{J^x} \to F'^{J^x} \times_{F'^{\partial J^x}} F^{\partial J^x}$$ is a fibration by Reedy fibrancy of $F \to F'$, and $$F^{L'} \to F'^{L'} \times_{F'^K} F^K$$ is a fibration by our inductive hypothesis.

If moreover $p$ is a trivial fibration, both of the previous maps are trivial fibrations (by characterization of Reedy trivial fibrations and by the inductive hypothesis respectively), so we can conclude.
\end{proof}
 
 As in the semi-simplicial case (\cite{ks2019internal}), we define a "cotensor" $K \triangleright F$ of a frame $F$ by a presheaf $K \in \mathbf{Set}^{J^{op}}$ by the formula $$(K \triangleright F)_x := F^{J^x \otimes K}$$
 and define an enrichment by $$\underline{Hom}_{\text{Fr}(\mathcal{T})}(F,F')_x : = Hom(F,J^x \triangleright F')$$
 
 We can now state the expected key theorem about categories of $J$-frames:
 
 \begin{theorem}
 \label[theorem]{J_frames_th}
  The category $J\text{-Fr}(\mathcal{T})$ of frames in $\mathcal{T}$ is enriched over $\mathbf{Set}^{J^{op}}$ and admits cotensors by finite presheaves on $J$. Moreover, the cotensors satisfy the required properties for $J\text{-Fr}(\mathcal{T})$ to be a $J$-tribe.
 \end{theorem}
 
 \begin{proof} 
 With the definition above, the proof follows by analogy with the semi-simplicial case (Theorem 3.7 in \cite{ks2019internal}). Explicitly, the properties expected from the cotensor hold by \cref{cotensor_reedy_cubes}, \cref{cotensor_reedy_cubes_gap}, and by the fact that a frame is a homotopical diagram.
 \end{proof}
 
 There is a canonical frame construction:
 
 \begin{corollary}
 \label[corollary]{jcan_frame}
 If $\mathcal{T}$ is a $J$-tribe, then there is a "canonical frame" tribe morphism $$c : \mathcal{T} \to J\text{-Fr}(\mathcal{T})$$ mapping $x \in \mathcal{T}$ to $i \mapsto x^{J^i}$. 
\end{corollary}

\begin{proof}
 As noted in \cref{mono_reedy}, the border inclusions $\partial J^i \to J^i$, for $i \in J$, are monomorphisms. Therefore, by the first condition of \cref{semicubical_tribe}, $x^{J^i} \to x^{\partial J^i}$ is a fibration, proving that the mapping takes values in Reedy fibrant diagrams. The second condition for the previous definition ensures that such diagrams are furthermore homotopically constant, and the third condition shows that the mapping $$c : \mathcal{T} \to J\text{-Fr}(\mathcal{T})$$ preserves anodyne maps. Since it is defined in terms of cotensors, it also preserves limits. Finally, it preserves fibrations by the first condition of \cref{semicubical_tribe}, concluding the proof that $c$ is a tribe morphism.
\end{proof}

The construction also works for $\pi$-tribes, as we now discuss. It has been established by Kapulkin and Lumsdaine that categories of homotopical inverse diagrams inherit $\Pi$-types (see \cite[Proposition 5.13]{kapulkin2021homotopical}, which is phrased in the context of categories with attributes).

\begin{proposition}[Adapted from {\cite[Proposition 5.14]{kapulkin2021homotopical}}]
\label[proposition]{htp_diag_trb}
 Let $\mathcal{T}$ be a $\pi$-tribe.
 \begin{enumerate}
  \item If $I$ is a homotopical inverse category with all arrows weak equivalences, then the full subcategory $\mathcal{T}_f^I \subset \mathcal{T}^I$ of Reedy fibrant homotopical diagrams admits a $\pi$-tribe structure.
  \item If $I \to J$ is a discrete opfibration between two inverse categories as above, then restriction along $I \to J$ yields a morphism of $\pi$-tribes $\mathcal{T}_f^J \to \mathcal{T}_f^I$.
 \end{enumerate}
\end{proposition}

\begin{proof}
 In \cite[Proposition 3.1]{cherradi2026generalized}, we rephrased the result by Kapulkin and Lumsdaine in terms of $\pi$-tribes in the case of non-homotopical diagrams. But the subcategory of homotopical diagrams is also closed under dependent product by \cite[Proposition 5.13]{kapulkin2021homotopical} (since $\Pi$-types in the corresponding category with attributes correspond to dependent product in the $\pi$-tribe). Therefore, $\mathcal{T}_f^I$ admits dependent products, and is itself a $\pi$-tribe.
 
 Similarly, the second point follows from the functoriality statement in \cite[Proposition 5.14]{kapulkin2021homotopical}, having in mind the alluded correspondence between $\Pi$-types in a category with attributes and the dependent product in a $\pi$-tribe. 
\end{proof}

\begin{lemma}
\label[lemma]{lift_zigzag}
 Let $\mathcal{F}$ be a fibration category, $I$ be a finite direct category, and $n$ a natural number. Write $I_n$ for the full subcategory of $I$ spanned by the objects of degree less than $n$.
 
 Then, given any Reedy fibrant diagrams $X : I^{op} \to \mathcal{F}$ and $Y : I_n^{op} \to \mathcal{F}$ and any zig-zag of (Reedy) trivial fibrations in $\mathcal{F}^{I_n^{op}}$ between $Y$ and the restriction $X_n$ of $X$ to $I_n^{op}$, there exists a lift of the whole zig-zag to $\mathcal{F}^{I^{op}}$.
\end{lemma}

\begin{proof}
There are two elementary cases to consider to lift a whole zig-zag. First, we may assume that we are given a trivial fibration $Y \to X_n$. As in the usual argument for constructing "Reedy" factorizations, we proceed by induction on the degree $k > n$ of the "missing" objects in $Y$, and we rely on the matching objects $M_i X$ for $i \in I$ of degree $k$ to define the lift $Y'$ of $Y$ at $i$ as $Y'(i) := M_i Y \times_{M_i X} X(i)$. The Reedy fibration criterion is then trivially satisfied for $Y'$ (the map $Y' \to X$ is a Reedy trivial fibration whose relative latching maps at $i$ are isomorphisms, for $\mathbf{deg}(i) > n$).

For the second case, we assume given a trivial fibration $X_n \to Y$. Then, we may define an extension $Y'$ such that $Y'(i) := X(i)$ for objects $i$ of degree strictly greater than $n$. The diagram so defined need not be fibrant, so we have to further factor the map $Y' \to *$ as a weak equivalence followed by a fibration, yielding a diagram $Y''$. Since the restriction of $Y'$ to $I_n^{op}$ is already Reedy fibrant (it coincides with $Y$), the factorization can be chosen so that $Y''$ also coincides with $Y$ when restricted to $I_n^{op}$.

This concludes the proof.
\end{proof}

\begin{proposition}
\label[proposition]{jframes_pi}
 Suppose that the degree function of $J$ takes values in the natural numbers, and that the full subcategories $J_n$ spanned by the objects of degree less than $n$ are finite.
 
 Then, if $\mathcal{T}$ is a $\pi$-tribe, so is the tribe $J\text{-Fr}(\mathcal{T})$.
\end{proposition}

\begin{proof}

 Using the first item of \cref{htp_diag_trb}, it is enough to check that homotopically constant diagrams in $\mathcal{T}^{J^{op}}_f$ are closed under dependent products.
 
 Write $\mathbf{Ho}_\infty(\mathcal{T})$ for the presented $(\infty,1)$-category.
 We have two towers of isofibrations
\[\begin{tikzcd}[ampersand replacement=\&]
	{\mathbf{Ho}_\infty(\mathcal{T}^{J^{op}})} \&\& {\mathbf{Ho}_\infty(\mathcal{T})^{\mathbf{N}(J^{op})}} \\
	{...} \&\& {...} \\
	{\mathbf{Ho}_\infty(\mathcal{T}^{J_1^{op}})} \&\& {\mathbf{Ho}_\infty(\mathcal{T})^{\mathbf{N}(J_1^{op})}} \\
	{\mathbf{Ho}_\infty(\mathcal{T}^{J_0^{op}})} \&\& {\mathbf{Ho}_\infty(\mathcal{T})^{\mathbf{N}(J_0^{op})}}
	\arrow[dashed, from=1-1, to=1-3]
	\arrow[from=1-1, to=2-1]
	\arrow[from=1-3, to=2-3]
	\arrow[two heads, from=2-1, to=3-1]
	\arrow[two heads, from=2-3, to=3-3]
	\arrow[from=3-1, to=3-3]
	\arrow[two heads, from=3-1, to=4-1]
	\arrow[two heads, from=3-3, to=4-3]
	\arrow[from=4-1, to=4-3]
\end{tikzcd}\]
where the horizontal solid arrows are equivalences of $(\infty,1)$-categories by \cite[Corollary 7.6.18]{cisinski2019} (using the Reedy fibration category structure on $\mathcal{T}^{J_n^{op}}$), so that we know every corresponding homotopy coherent diagram can be rigidified up to weak equivalence.
Now, consider homotopically constant diagrams $X$ and $Y$ in $\mathcal{T}$, and let us check that the exponential $X^Y$ is homotopically constant (for notational simplicity, we consider exponentials but the argument works for dependent products in general). We write $Z_n$ for the restriction of a diagram $Z \in \mathcal{T}^{J^{op}}$ to $J_n^{op}$.

The assumption that $X$ and $Y$ are homotopically constant implies the same for $X_n$ and $Y_n$. Thanks to the above observation (also applied to $J^\triangleleft$), we see that the latter is equivalent to the existence of "essential" extensions $X'_n$ and $Y'_n$ to the inverse category $(J_n^\triangleleft)^{op}$; that is, $X_n'$ and $Y_n'$ restrict to $J_n^{op}$ to yield diagrams equivalent to $X_n$ and $Y_n$. By \cref{lift_zigzag}, we can lift those to "real" extensions of $X_n$ and $Y_n$ (still denoted $X_n'$ and $Y_n'$).

By the first item of \cref{htp_diag_trb}, $X_n'^{Y_n'}$ is a homotopical diagram, and it restricts to $J$ to give a diagram isomorphic to $X_n^{Y_n}$, by the second item of \cref{htp_diag_trb}, since $J_n \subset J_n^\triangleright$ is a discrete fibration. Hence, $X_n^{Y_n}$ is also homotopically constant. 

Finally, we have homotopical diagrams $X_n^{Y_n}$ (which coincide with $(X^Y)_n$ since $J_n \subset J$ is a discrete fibration, again by \cref{htp_diag_trb}). By \cref{lift_zigzag}, we can inductively build a consistent sequence $(X^Y)'_n$ of diagrams, which together yield a diagram $(X^Y)' : (J^\triangleright)^{op} \to \mathcal{T}$ which extends $X^Y$. This means that $X^Y$ is homotopically constant, concluding the proof that the tribe $J\text{-Fr}(\mathcal{T})$ is a $\pi$-tribe.
\end{proof}
 
 \section{Generalized inverse diagrams in a tribe}
 
 From now on, $R$ is a generalized direct category (in particular, an EZ-category), in the sense of the following definition:
 
 \begin{definition}
 We say that a category $R$ is a generalized direct category when there exists a degree function $\mathbf{deg} : \mathbf{Ob}(R) \to \alpha$ for some ordinal $\alpha$ such that invertible morphisms preserve the degree, and non-invertible ones strictly raise it. Equivalently, this asks for the existence of a conservative functor $R \to \alpha$ (where $\alpha$ is seen as a posetal category).
 
 We say that $R$ is a generalized inverse category when $R^{op}$ is generalized direct.
\end{definition}
 
We assume that $R$ is monoidal, has finitely many objects and isomorphisms in each degree, and that it is equipped with a functor into it satisfying the condition below, introduced in \cite{cherradi2026generalized}:
 
 \begin{itemize}
 \item There is a degree-preserving monoidal functor $r : R_0 \to R$ from a strictly direct monoidal category, such that every arrow $f : a \to b$ in $R$ lifts to an arrow $k : x \to y$ in $R_0$ up to isomorphism, i.e., such that there are isomorphisms $w : a \simeq r(x)$ and $w' : b \simeq r(y)$ fitting in a commutative square:

\[\begin{tikzcd}[ampersand replacement=\&]
	a \&\& b \\
	\\
	{r(x)} \&\& {r(y)}
	\arrow["f", from=1-1, to=1-3]
	\arrow["w"', from=1-1, to=3-1]
	\arrow["{w'}", from=1-3, to=3-3]
	\arrow["{r(k)}"', from=3-1, to=3-3]
\end{tikzcd}\]
(in other words, the functor $r^\rightarrow : R_0^\rightarrow \to R^\rightarrow$ is required to be essentially surjective on objects).
\end{itemize}

The category of semi-cubes is our main example of such a category, and we will be able to apply directly the content of this section to that case (see Section 8).
 
  Recall the following definition (\cite[Definition 3.1]{cherradi2026generalized}):

 \begin{definition}
 \label[definition]{def_dr}
 Write $\mathbf{F}$ for the free category comonad on $\mathbf{Cat}$ stemming from the canonical adjunction with reflexive graphs.
 
 Form the pushout as in the diagram below:
\[\begin{tikzcd}[ampersand replacement=\&]
	{\mathbf{F}(R_0)} \&\& {\mathbf{F}(R)} \& \\
	\\
	{R_0} \&\& {\mathbf{F}_\simeq(R)} \\
	\&\&\& R
	\arrow[from=1-1, to=1-3]
	\arrow[from=1-1, to=3-1]
	\arrow["\iota", from=1-3, to=3-3]
	\arrow[curve={height=-6pt}, from=1-3, to=4-4]
	\arrow["{\iota_0}"', from=3-1, to=3-3]
	\arrow[curve={height=6pt}, from=3-1, to=4-4]
	\arrow["\ulcorner"{anchor=center, pos=0.125, rotate=180}, draw=none, from=3-3, to=1-1]
	\arrow["{p_0}", dashed, from=3-3, to=4-4]
\end{tikzcd}\]
 Finally, write $\mathbf{D}_r$ for the full subcategory of the twisted arrow category $\mathbf{Tw}(\mathbf{F}_\simeq(R))$ spanned by the objects consisting of an arrow $x \to y$ of the form $x \to z \to y$, where $x \to z$ is the image of a map in $R_0$ (possibly an identity map) through $\iota_0$, and $z \to y$ is an identity arrow or corresponds to a "free" isomorphism in $R$  (that is, an arrow obtained by taking the image under $\iota :\mathbf{F}(R) \to \mathbf{F}_\simeq(R)$ of a path of length one in $R$ whose underlying arrow is an isomorphism in $R$). 
 
 It comes with a projection $p : \mathbf{D}_r \to R$, which projects into the codomain.
 \end{definition}
  
  In \cite{cherradi2026generalized}, we established the following:
  
\begin{lemma}
  $\mathbf{D}_r$ inherits a direct category structure.
\end{lemma}
 
 \begin{lemma}
  The canonical functor $p : \mathbf{D}_r \to R$ is absolutely dense (i.e., the precomposition functor $p^* : \mathbf{Set}^{R} \to \mathbf{Set}^{\mathbf{D}_r}$ is fully faithful).
 \end{lemma}
  
 \section{Generalized frames in a tribe}
 
 We now fix a tribe $\mathcal{T}$. Our goal is to carve out a subcategory of fibrant diagrams of $\mathcal{T}^R$ that enjoys a tribe structure.
  
 \begin{definition}
  We define the class of $p$-fibrations as the class of morphisms $m : F \to F'$ in $\mathcal{T}^{R^{op}}$ such that $p^* m$ is a Reedy fibration in $\mathcal{T}^{\mathbf{D}_r^{op}}$.
 \end{definition}
 
 \begin{definition}
  We define an $R$-frame in the tribe $\mathcal{T}$ to be a diagram mapped by $p^*$ to a $\mathbf{D}_r$-frame. We write $R\text{-Fr}(\mathcal{T})$ for the full subcategory of $\mathcal{T}^{R^{op}}$ spanned by the $R$-frames. We define the fibrations between two such diagrams to be the $p$-fibrations.
 \end{definition}
 
 It turns out that this definition yields frames in the expected sense:
 \begin{proposition}
  A diagram $R^{op} \to \mathcal{T}$ is an $R$-frame if and only if it is a $p$-fibrant homotopical diagram, where the category $R$ is given the homotopical structure with all maps weak equivalences.
 \end{proposition} 
 
 \begin{proof}
  If a diagram $X$ in $ \mathcal{T}^{R^{op}}_f$ is homotopically constant, then so is $p^*X$ by functoriality. The converse follows directly from \cref{constant_diagrams} below.
 \end{proof}

 It is well-known that left Kan extensions along opfibrations can be computed pointwise directly from fibers rather than from "lax" fibers. More generally, the following lemma is folklore:
 
 \begin{lemma}
 \label[lemma]{pullback_fib_exact}
  A pullback of categories  

\[\begin{tikzcd}[ampersand replacement=\&]
	{B \times_A C} \&\& B \\
	\\
	C \&\& A
	\arrow[from=1-1, to=1-3]
	\arrow[from=1-1, to=3-1]
	\arrow["p", from=1-3, to=3-3]
	\arrow["id"{description}, between={0.3}{0.7}, Rightarrow, from=3-1, to=1-3]
	\arrow["f"', from=3-1, to=3-3]
\end{tikzcd}\]
where $p : B \to A$ is a Grothendieck fibration, is an exact square.
 \end{lemma}
 
 \begin{proof}
  Consider an object $c$ of $C$. It will be enough to check that composing on the left with the comma square below induces an exact square,

\[\begin{tikzcd}[ampersand replacement=\&]
	{(f^*p)^{-1}(c)} \& {c \downarrow f^*p} \&\& {B \times_A C} \&\& B \\
	\\
	\& {*} \&\& C \&\& A
	\arrow[""{name=0, anchor=center, inner sep=0}, dashed, from=1-1, to=1-2]
	\arrow[from=1-1, to=3-2]
	\arrow[""{name=1, anchor=center, inner sep=0}, "G"', shift right=3, curve={height=12pt}, from=1-2, to=1-1]
	\arrow[from=1-2, to=1-4]
	\arrow[from=1-2, to=3-2]
	\arrow[from=1-4, to=1-6]
	\arrow["{f^*p}"', from=1-4, to=3-4]
	\arrow["p", from=1-6, to=3-6]
	\arrow[between={0.3}{0.7}, Rightarrow, from=3-2, to=1-4]
	\arrow["c"', from=3-2, to=3-4]
	\arrow["id"{description}, between={0.3}{0.7}, Rightarrow, from=3-4, to=1-6]
	\arrow["f"', from=3-4, to=3-6]
	\arrow["\dashv"{anchor=center, rotate=99}, draw=none, from=0, to=1]
\end{tikzcd}\]
where $f^*p$ is a fibration as a pullback of one.

But the canonical map $(f^*p)^{-1}(c) \to c \downarrow f^*p$ admits a right adjoint $G$ (sending $c \to f^*p(x)$ to the domain of a cartesian lift). Therefore, $(f^*p)^{-1}(c) \to c \downarrow f^*p$ is an initial functor as a left adjoint, and the following square is exact:
\[\begin{tikzcd}[ampersand replacement=\&]
	{(f^*p)^{-1}(c)} \&\& {c \downarrow f^*p} \\
	\\
	{*} \&\& {*}
	\arrow[from=1-1, to=1-3]
	\arrow[from=1-1, to=3-1]
	\arrow[from=1-3, to=3-3]
	\arrow[between={0.3}{0.7}, Rightarrow, from=3-1, to=1-3]
	\arrow[from=3-1, to=3-3]
\end{tikzcd}\]

Hence, it is equivalent to establish that the outer rectangle in the first diagram is exact.
Again, since $p$ is a fibration by assumption, the functor $p^{-1}(fc) \to fc \downarrow p$ is initial so that the composite
\[\begin{tikzcd}[ampersand replacement=\&]
	{p^{-1}(fc) \simeq (f^*p)^{-1}(c)} \&\& {fc \downarrow p} \&\& B \\
	\\
	\&\& {*} \&\& A
	\arrow[""{name=0, anchor=center, inner sep=0}, dashed, from=1-1, to=1-3]
	\arrow[from=1-1, to=3-3]
	\arrow[""{name=1, anchor=center, inner sep=0}, shift right=3, curve={height=12pt}, from=1-3, to=1-1]
	\arrow[from=1-3, to=1-5]
	\arrow[from=1-3, to=3-3]
	\arrow["p", from=1-5, to=3-5]
	\arrow[between={0.3}{0.7}, Rightarrow, from=3-3, to=1-5]
	\arrow["fc"', from=3-3, to=3-5]
	\arrow["\dashv"{anchor=center, rotate=113}, draw=none, from=0, to=1]
\end{tikzcd}\]
is exact. We are therefore able to conclude.
 \end{proof}
 
We write $(\mathbf{D}_r)_{\leq n}$ for the full subcategory of $\mathbf{D}_r$ spanned by the arrows whose codomain's degree is at most $n$.
 
 \begin{corollary}
 \label[corollary]{constant_diagrams}
  The functor $$p_* : \mathcal{T}^{\mathbf{D}_r^{op}}_f \to \mathcal{T}^{R^{op}}$$ maps Reedy fibrant diagrams that are homotopically constant to homotopically constant diagrams.
 \end{corollary}
 
 \begin{proof}
  Let $X$ be a homotopically constant Reedy fibrant diagram of shape $\mathbf{D}_r^{op}$. It is enough to prove that the restriction of $p_* X$ to the full subcategory $R^{op}_{\leq n}$ of $R^{op}$ spanned by the objects of degree at most $n$ is homotopically constant, for every natural number $n$. Fix such a natural number $n$.
  We have the following pullback square, which is also an exact square (as a pullback along a Grothendieck fibration by \cref{pullback_fib_exact}):

\[\begin{tikzcd}[ampersand replacement=\&]
	{(\mathbf{D}_r)_{\leq n}} \&\& {\mathbf{D}_r} \\
	\\
	{R_{\leq n}} \&\& R
	\arrow[from=1-1, to=1-3]
	\arrow["{p_n}"', from=1-1, to=3-1]
	\arrow["\ulcorner"{anchor=center, pos=0.125}, draw=none, from=1-1, to=3-3]
	\arrow["id"{description}, between={0.3}{0.7}, Rightarrow, from=1-3, to=3-1]
	\arrow["p", from=1-3, to=3-3]
	\arrow[from=3-1, to=3-3]
\end{tikzcd}\]

 We may further consider the following comma square,

\[\begin{tikzcd}[ampersand replacement=\&]
	{r \downarrow p_n} \&\& {(\mathbf{D}_r)_{\leq n}} \\
	\\
	{*} \&\& {R_{\leq n}}
	\arrow["{q_n}", from=1-1, to=1-3]
	\arrow["{!}"', from=1-1, to=3-1]
	\arrow["{p_n}", from=1-3, to=3-3]
	\arrow[between={0.3}{0.7}, Rightarrow, nfold, from=3-1, to=1-3]
	\arrow["r"', from=3-1, to=3-3]
\end{tikzcd}\]
where $r \downarrow p_n$ is finite and admits an initial object (given by the identity arrow at $r$). In particular, $r \downarrow p_n$ has a contractible nerve.

This implies that the $\infty$-functor $$\mathbf{lim} : \mathbf{Ho}_\infty(\mathcal{T})^{(r \downarrow p_n)^{op}} \to \mathbf{Ho}_\infty(\mathcal{T})$$ maps a (homotopically) constant diagram, say at $x \in \mathcal{T}$, to $x$. In other words, the following square is exact:
\[\begin{tikzcd}[ampersand replacement=\&]
	{r \downarrow p_n} \&\& {*} \\
	\\
	{*} \&\& {*}
	\arrow[from=1-1, to=1-3]
	\arrow[from=1-1, to=3-1]
	\arrow[from=1-3, to=3-3]
	\arrow["id"{description}, between={0.3}{0.7}, Rightarrow, nfold, from=3-1, to=1-3]
	\arrow[from=3-1, to=3-3]
\end{tikzcd}\]

In particular, if $X \in \mathcal{T}^{(\mathbf{D}_r)_{\leq n}^{op}}_f$ is a homotopically constant diagram (at some $x \in \mathcal{T}$), so is its restriction $q_n^* X$, so that the homotopy limit $l_X$ of the latter diagram is equivalent to $x$. This shows that the homotopy right Kan extension of $X$ along $p_n$ is also constant at $x$. Since the right Kan extension functor $(p_n)_*$ computes the correct homotopy right Kan extension for fibrant diagrams, we can conclude.
\end{proof}

 \begin{theorem}
 \label[theorem]{R_tribe}
  $R\text{-Fr}(\mathcal{T})$ enjoys the structure of a tribe.
 \end{theorem}
 
 \begin{proof}
  We already know from \cite[Theorem 2.3]{cherradi2026generalized} that the subcategory of $p$-fibrant diagrams in $\mathcal{T}^{R^{op}}$ admits a tribe structure with the fibrations being the $p$-fibrations. We still need to check that this restricts to a tribe structure on $R\text{-Fr}(\mathcal{T})$.
  
  The class of $p$-fibrations is stable under pullback (using \cref{constant_diagrams}) and composition. Moreover, pointwise anodyne maps are stable under pullback along $p$-fibrations because the pointwise anodyne maps in $\mathcal{T}^{\mathbf{D}_r^{op}}_f$ are the anodyne maps by \cite[Lemma 2.22]{ks2019internal}
  Overall, this means that $R\text{-Fr}(\mathcal{T})$ also enjoys a tribe structure, with the fibrations the $p$-fibrations.
\end{proof}

The construction also works with $\pi$-tribes:

\begin{proposition}
\label[proposition]{pitribe_frames}
 Assume that $R$ satisfies the assumptions made on $J$ in \cref{jframes_pi}.
 If $\mathcal{T}$ is a $\pi$-tribe, then so is $R\text{-Fr}(\mathcal{T})$.
\end{proposition}

\begin{proof}

 Given two $p$-fibrations $f : A \to B$ and $g : B \to C$, we claim that $p_* \Pi_{p^* g}(p^* f)$, which defines an internal product $\Pi_g f$ of $f$ along $g$, is an $R$-frame. But $\Pi_{p^* g}(p^* f)$ is homotopically constant since $A$, $B$ and $C$ are (applying \cref{jframes_pi}, since under the assumptions made on $R$, the same hold for $\mathbf{D}_r$). Therefore, the claim follows directly by definition of $R$-frames and by construction of the dependent product in $\mathcal{T}^{R^{op}}_f$.

This is enough to conclude that $R\text{-Fr}(\mathcal{T})$ forms a $\pi$-tribe.
\end{proof}

\section{Enrichment for \texorpdfstring{$R$}{\textbf{R}}-tribes}

In what follows, we assume that $R_0$ and $R$ are endowed with the structures of monoidal categories, and that the functor $R_0 \to R$ is (strong) monoidal.

\begin{definition}
 We equip $\mathbf{D}_r$ with a monoidal product as follows. Given two objects of $\mathbf{D}_r$ represented by arrows $x \to y$ and $x' \to y'$ in $R$ possibly followed by "free" isomorphisms $w : y \to z$ and $w' : y' \to z'$, we define their monoidal product as the arrow $x \otimes x' \to y \otimes y'$ in $R_0$ followed by the automorphism $w \otimes w'$ (if both lists of isomorphisms are empty, then we take the empty list). The axioms for a monoidal category then follow from those for the monoidal structure on $R_0$ and $R$.
\end{definition}

Note that the functor $p$ is monoidal with this definition of monoidal structure on $\mathbf{D}_r$.

 \begin{lemma}
 \label[lemma]{strong_lan}
 The functor $$p_! : \mathbf{Set}^{\mathbf{D}_r^{op}} \to \mathbf{Set}^{R^{op}}$$ is canonically a (strong) monoidal functor.
 \end{lemma}
 
\begin{proof}
   In terms of profunctors, presheaves $K$ and $L$ on $\mathbf{D}_r$ yield a profunctor $K \overline{\otimes} L : \mathbf{D}_r^{op} \times \mathbf{D}_r^{op} \to *$ (corresponding to the external tensor product). The Day convolution process then corresponds to precomposition with the representable profunctor $\mathbf{D}_r^{op}(\otimes, *) : \mathbf{D}_r^{op} \to \mathbf{D}_r^{op} \times \mathbf{D}_r^{op}$. In particular, we have that $p_!(K \otimes L)$ corresponds to the top composite in the diagram below,

\[\begin{tikzcd}[ampersand replacement=\&]
	{R^{op} } \&\& {\mathbf{D}_r^{op} } \&\& {\mathbf{D}_r^{op} \times \mathbf{D}_r^{op}} \&\& {*} \\
	\\
	\&\& {R^{op}  \times R^{op} }
	\arrow["{Hom(p, *)}", from=1-1, to=1-3]
	\arrow["{Hom(\otimes, *)}"', from=1-1, to=3-3]
	\arrow["{Hom(\otimes, *)}", from=1-3, to=1-5]
	\arrow["{K \overline{\otimes} L}", from=1-5, to=1-7]
	\arrow["{Hom(p \times p, *)}"', from=3-3, to=1-5]
\end{tikzcd}\]
which, by functoriality, coincides with $p_!(K) \otimes p_!(L)$.
\end{proof}
 
As a corollary, the adjunction $p_! \vdash p^*$ is a monoidal adjunction, where $p^*$ inherits the structure of a lax monoidal functor.

\begin{lemma}
\label[lemma]{change_enrichment}
 Let $\mathcal{T}$ be an $R$-tribe. Then $\mathcal{T}$ is also canonically enriched over $\mathbf{Set}^{\mathbf{D}_r^{op}}$, making it a $\mathbf{D}_r$-tribe.
\end{lemma}

\begin{proof}
 By \cref{strong_lan}, the left Kan extension functor $$p_! : \mathbf{Set}^{\mathbf{D}_r^{op}} \to \mathbf{Set}^{R^{op}}$$ is strong monoidal with respect to the monoidal structure obtained by Day convolution, hence the precomposition functor $p^*$ is lax monoidal.
 
 We rely on the change-of-enrichment provided by $p^*$ (i.e., we define the hom-object of $\mathcal{T}$ to be the image under $p^*$ of the hom-object for the original enrichment over $\mathbf{Set}^{R^{op}}$).
 
 Suppose $K$ is a finite presheaf over $\mathbf{D}_r$. It follows from the pointwise formula for computing the left Kan extension functor $p_!$ that $p_! K$ is a finite presheaf over $R$, and the following natural isomorphisms prove the existence of cotensors by $K$:
 
 \begin{align*}  
 Hom_{\mathbf{Set}^{\mathbf{D}_r^{op}}}(K,p^* Hom_\mathcal{T}(x,y)) &\simeq &Hom_{\mathbf{Set}^{R^{op}}}(p_! K,Hom_\mathcal{T}(x,y))\\
 & \simeq &Hom_\mathcal{T}(x,y^{p_! K})
 \end{align*}
 
 Similarly, observe that any monomorphism $K \to L$ between finite presheaves over $\mathbf{D}_r$ is mapped by $p_!$ to a monomorphism between finite presheaves over $R$. Indeed, for every object $r \in R$, the comma category $r \downarrow p$, as in the following diagram,

\[\begin{tikzcd}[ampersand replacement=\&]
	{r \downarrow p} \&\& {\mathbf{D}_r} \\
	\\
	{*} \&\& R
	\arrow[from=1-1, to=1-3]
	\arrow[from=1-1, to=3-1]
	\arrow["p", from=1-3, to=3-3]
	\arrow[between={0.3}{0.7}, Rightarrow, from=3-1, to=1-3]
	\arrow["r"', from=3-1, to=3-3]
\end{tikzcd}\]

has an initial object (namely the object given by the identity arrow $id_r : r \to r$, thought of as an object of $\mathbf{D}_r$, together with $id_r$ as the component $p(id_r) \to r$). This means, in particular, that the category $r \downarrow p$ is cofiltered, so that the functor $p$ is (representably) flat; thus, the left Kan extension functor $$p_! : \mathbf{Set}^{\mathbf{D}_r^{op}} \to \mathbf{Set}^{R^{op}}$$
preserves finite limits. In particular, $p_!$ preserves monomorphisms, as required.

 Therefore, it follows that the cotensor by such morphisms is a fibration in $\mathcal{T}$. Finally, since $p_!$ maps representable presheaves to representable ones, the cotensor by any representable monomorphism $\mathbf{D}_r^x \to \mathbf{D}_r^y$ is always a weak equivalence (hence a trivial fibration). This concludes the proof that $\mathcal{T}$ is canonically a $\mathbf{Set}^{\mathbf{D}_r^{op}}$-tribe. 
\end{proof}

\begin{corollary}
\label[corollary]{rcan_frame}
 If $\mathcal{T}$ is an $R$-tribe, then there is a "canonical frame" tribe morphism $$c : \mathcal{T} \to R\text{-Fr}(\mathcal{T})$$ mapping $x \in \mathcal{T}$ to $r \mapsto x^{R^r}$. 
\end{corollary}

\begin{proof}
 Observe that $c$ factors through the "canonical frame" functor $$c' : \mathcal{T} \to \mathcal{T}^{\mathbf{D}_r^{op}}$$ resulting from the $\mathbf{D}_r$-tribe structure on $\mathcal{T}$ (by \cref{change_enrichment}). Precisely, we have a commutative triangle 
\[\begin{tikzcd}[ampersand replacement=\&]
	{\mathcal{T}} \&\&\&\& {R\text{-Fr}(\mathcal{T})} \\
	\\
	\&\& {\mathbf{D}_r\text{-Fr}(\mathcal{T})}
	\arrow["c", from=1-1, to=1-5]
	\arrow["{c'}"', from=1-1, to=3-3]
	\arrow["{p_*}"', from=3-3, to=1-5]
\end{tikzcd}\]
and we conclude from the fact that $c'$ is a tribe morphism (by \cref{jcan_frame}), but also that $p_*$ restricts to a tribe morphism between tribes of frames (this follows directly from the non-homotopical version, established in \cite[Proposition 2.2]{cherradi2026generalized}).
\end{proof}

\section{The evaluation functor}
 
 When $R$ has a unique object of degree $0$, we have a canonical evaluation functor: $$R\text{-Fr}(\mathcal{T}) \to \mathcal{T}$$
 We would like to provide sufficient conditions under which this functor establishes a DK-equivalence. Under a contractibility assumption, we first get the following result, analogous to \cite[Proposition 3.9]{schwede2013p}:
 
 \begin{lemma}
 \label[lemma]{cone_frame}
 Suppose that $(\mathbf{D}_r)_{\leq n}$ has a contractible nerve for every natural number $n$. Also assume that $\mathbf{D}_r$ has only one object, denoted by $0$, with codomain an object of $R$ of degree $0$ (this is true if $R$ has only one object of degree $0$ admitting no automorphism).
  Consider a $\mathbf{D}_r$-frame $F$ in $\mathcal{T}$ and a morphism $x \to F_0$. Then, the latter can be extended to a cone $\Delta x \to F$.
 \end{lemma}
 
 \begin{proof}
  We will inductively extend the cone to the full subcategories $\mathbf{D}_{R,\leq n}$ of $\mathbf{D}_r$, which are finite categories with contractible nerves. For $n = 0$, $\mathbf{D}_{R,\leq n}$ has only one element $0$ by assumption, and we already have a cone as part of the starting input data.

  Now, consider the functor $q_n$ from $\mathbf{D}_{R,\leq n+1}$ to the walking arrow category $\rightarrow$, mapping $\mathbf{D}_{R,\leq n}$ to the source, and mapping the remaining objects in $\mathbf{D}_{R,\leq n+1}$ to the target. We have the following comma square (which is also a pullback),
\[\begin{tikzcd}[ampersand replacement=\&]
	{\mathbf{D}_{\square^s_\sharp,\leq n}} \&\& {\mathbf{D}_{\square^s_\sharp,\leq n+1}} \\
	\\
	{*} \&\& \rightarrow
	\arrow[from=1-1, to=1-3]
	\arrow[from=1-1, to=3-1]
	\arrow["id"{description}, between={0.3}{0.7}, Rightarrow, from=1-3, to=3-1]
	\arrow["{q_n}", from=1-3, to=3-3]
	\arrow["s"', from=3-1, to=3-3]
\end{tikzcd}\]
so we can conclude, applying \cite[Theorem 9.4.3 (2)(b)]{radulescu2006cofibrations}, because $F$ is fibrant (hence so is its restriction to $\mathbf{D}_{R,\leq n+1}$), that the diagram $(q_n)^{op}_* F$ is a fibrant diagram (i.e., a fibration) with source the limit of the restriction of $F$ to $\mathbf{D}_{R,\leq n+1}$, and with target the limit of its restriction to $\mathbf{D}_{R,\leq n}$. This arrow is not only a fibration but moreover also a weak equivalence since both corresponding diagrams are homotopically constant and of contractible shape. As a trivial fibration in a tribe, it admits a section by \cite[Corollary 3.4.7]{joyal2017notes}. This section allows us to extend the cone from $x$ at stage $n$ to stage $n+1$, concluding the proof.
 \end{proof}
 
 The following result adapts Theorem 3.10 of \cite{schwede2013p}:
 \begin{proposition}
 \label[proposition]{ev0_funct}
  Under the assumptions of \cref{cone_frame}, and provided that $R$ has only one object $0$ of degree $0$, the functor $$\mathbf{ev}_0 : R\text{-Fr}(\mathcal{T}) \to \mathcal{T}$$
  mapping a frame $F$ to its value at $0$ is a DK-equivalence.
 \end{proposition}

 \begin{proof}
  Since $R\text{-Fr}(\mathcal{T})$ and $\mathcal{T}$ are, in particular, fibration categories, we can use the characterization of DK-equivalences between fibration categories by means of the approximation property introduced by Cisinski in \cite{cisinski2010categories}. The first condition $(AP1)$ is easily checked: the assumption on $\mathbf{D}_r$ (contractibility of $(\mathbf{D}_r)_{\leq n}$) implies connectedness of $\mathbf{D}_r$, and thus of $R$, so that a map in $R\text{-Fr}(\mathcal{T})$ is a pointwise weak equivalence if and only if its restriction to the object $0$ is one. For the second, consider an arrow $x \to F_0 = \mathbf{ev}_0 F$, where $F$ is a frame. We use \cref{cone_frame} to provide a cone $\Delta x \to p^* F$. We can now take a factorization of $\Delta x \to p^* F$ as a weak equivalence followed by a Reedy fibration $F' \to p^* F$. This provides a diagram

\[\begin{tikzcd}[ampersand replacement=\&]
	x \&\& {\mathbf{ev}_0 F} \\
	\\
	x \&\& {\mathbf{ev}_0 (p_* F')}
	\arrow[from=1-1, to=1-3]
	\arrow["{id_x}", from=3-1, to=1-1]
	\arrow["\sim"{description}, from=3-1, to=3-3]
	\arrow[from=3-3, to=1-3]
\end{tikzcd}\]
where the map $x \to \mathbf{ev}_0 (p_* F')$ is a weak equivalence because it is obtained by restriction: in the corresponding computation for pointwise Kan extensions, following from the exact square below,
\[\begin{tikzcd}[ampersand replacement=\&]
	{p \downarrow 0} \&\& {\mathbf{D}_r} \\
	\\
	{*} \&\& R
	\arrow[from=1-1, to=1-3]
	\arrow[from=1-1, to=3-1]
	\arrow[between={0.3}{0.7}, Rightarrow, from=1-3, to=3-1]
	\arrow["p", from=1-3, to=3-3]
	\arrow["0"', from=3-1, to=3-3]
\end{tikzcd}\]
the category $p \downarrow 0$ is the terminal category.

This proves that $(AP2)$ also holds.
 \end{proof}
 
 \begin{remark}
  The previous result states that a subcategory of homotopical diagrams of shape $R$ in the $(\infty,1)$-category $\mathbf{Ho}_\infty (\mathcal{T})$ presented by $\mathcal{T}$ is equivalent to $\mathbf{Ho}_\infty (\mathcal{T})$. For the subcategory of all homotopical diagrams, this should reasonably hold (for arbitrary $\mathcal{T}$) only when $R$ has a contractible nerve. In particular, $R\text{-Fr}(\mathcal{T})$ need not present the $(\infty,1)$-category $\mathbf{Ho}_\infty (\mathcal{T})^R$. This is the case for semi-cubical frames, discussed in the last section, since the category of semi-cubes we consider does not have a contractible nerve.
 \end{remark}
 
 \section{The \texorpdfstring{$R$}{\textbf{R}}-tribe of \texorpdfstring{$R$}{\textbf{R}}-frames}
 
  In this section, we aim at providing a suitable enrichment to our tribe of $R$-frames, so as to make it an $R$-tribe.
 
  For $K$ a finite presheaf on $R$, and $F$ an $R$-frame, we also write $F^K$ for a representing object of the functor $\mathcal{T}^{op} \to \mathbf{Set}$ mapping $z$ to $Hom_{\mathbf{Set}^{R^{op}}}(K,\mathcal{T}(z,F))$, provided such an object exists.
 
 The results from Section 2 are easily transferred to this generalized setting.
 
 \begin{lemma}
\label[lemma]{g_cotensor_reedy_cubes}
 Consider a tribe $\mathcal{T}$ and an $R$-frame $F$ in $\mathcal{T}$.
 For any finite presheaf $K \in  \mathbf{Set}^{R^{op}}$, $F^K$ exists in $\mathcal{T}$. Moreover, the functor $F \mapsto F^K$ takes any  $p$-fibration (resp. trivial $p$-fibration) between frames to a fibration (resp. a trivial fibration), and it takes anodyne maps to anodyne maps.
\end{lemma}

\begin{proof}
 We have a natural isomorphism
 $$Hom_{\mathbf{Set}^{R^{op}}}(K,\mathcal{T}(z,F)) \simeq Hom_{\mathbf{Set}^{\mathbf{D}_r^{op}}}(p^*K,p^*\mathcal{T}(z,F))$$
 so that the existence of the representing object, which coincides with $(p^* F)^{p^* K}$, follows from \cref{cotensor_reedy_cubes}, as well as the stated properties (observing that $p^*$ preserves monomorphisms).
\end{proof}

\begin{lemma}
 \label[lemma]{g_cotensor_reedy_cubes_gap}
  For any monomorphism $i : K \to L$ between finite presheaves on $R$, and any fibration $q : F \to F'$ between $R$-frames, the gap map in the diagram below

\[\begin{tikzcd}[ampersand replacement=\&]
	{F^L} \\
	\& {F'^L \times_{F'^K} F^K} \&\& {F^K} \\
	\\
	\& {F'^L} \&\& {F'^K}
	\arrow[dashed, from=1-1, to=2-2]
	\arrow[curve={height=-12pt}, from=1-1, to=2-4]
	\arrow[curve={height=12pt}, from=1-1, to=4-2]
	\arrow[from=2-2, to=2-4]
	\arrow[from=2-2, to=4-2]
	\arrow["\ulcorner"{anchor=center, pos=0.125}, draw=none, from=2-2, to=4-4]
	\arrow[from=2-4, to=4-4]
	\arrow[from=4-2, to=4-4]
\end{tikzcd}\]
  is a fibration in $\mathcal{T}$, that is moreover a trivial fibration whenever $q$ is so.
 \end{lemma}
 
 \begin{proof}
  Just like in the proof of \cref{g_cotensor_reedy_cubes}, this follows from the corresponding property for $p^*F \to p^*F'$ since $p^*$ maps monomorphisms to monomorphisms, which is provided by \cref{cotensor_reedy_cubes_gap}.
 \end{proof}
 
 We may define an enrichment of $R\text{-Fr}(\mathcal{T})$ in $\mathbf{Set}^{R^{op}}$ constructed as for \cref{J_frames_th}. Namely, we define a "cotensor" $K \triangleright F$ of a frame $F$ by a presheaf $K \in \mathbf{Set}^{R^{op}}$ by the formula $$(K \triangleright F)_x := F^{R^x \otimes K}$$
 and define an enrichment by $$\underline{Hom}_{\text{Fr}(\mathcal{T})}(F,F')_x : = Hom(F, R^x \triangleright F')$$
 
 \begin{theorem}
 \label[theorem]{R_frames_th}
  Equipped with the previous enrichment over $\mathbf{Set}^{R^{op}}$, the category $R\text{-Fr}(\mathcal{T})$ of $R$-frames in $\mathcal{T}$ admits cotensors by finite presheaves on $R$. Moreover, the cotensors satisfy the required properties for $R\text{-Fr}(\mathcal{T})$ to be an $R$-tribe.
 \end{theorem}
 
 \begin{proof}
 
 Using the definition above, the properties expected from the cotensor hold by \cref{g_cotensor_reedy_cubes_gap}, \cref{g_cotensor_reedy_cubes} and by the fact that a frame is a homotopical diagram.
 
 This makes $R\text{-Fr}(\mathcal{T})$ an $R$-tribe, with the tribe structure provided by \cref{R_tribe}.
 \end{proof}

\section{Semi-cubical tribes}

 We now come to the application of our work above, i.e., we consider a category of semi-cubes. 
 Explicitly, we consider the free monoidal category $(\square_\sharp, \otimes, I^0)$ generated by two face maps $\delta_0, \delta_1 : I^0 \to I^1$ with domain the monoidal unit. This is a subcategory of the cube category $\mathbb{I}$ introduced in section 4 of \cite{grandis2003cubical}, excluding degeneracies. Thus, the objects of $\square_\sharp$ are of the form $I^n := I^1 \otimes ... \otimes I^1$ and can therefore be identified with the natural numbers (we may write $[n]$ by analogy to the semi-simplex category). The morphisms are generated from the two face maps $\delta_0$ and $\delta_1$ (and the identity maps) under the monoidal product. 
 We will be interested in a symmetric version: the free symmetric monoidal category $(\square_\sharp^s, \otimes, I^0)$ generated by two face maps $\delta_0, \delta_1 : I^0 \to I^1$ together with an involution of the interval (a reversal) $\rho: I^1 \to I^1$.
 
 There is a monoidal functor $r_\square : \Delta_{a,\sharp} \to \square_\sharp^s$ from the augmented semi-simplex category to the category of symmetric cubes we consider, defined by mapping $[n] \in \Delta_{a,\sharp}$ to $I^{n+1}$, and the unique map $[-1] \to [0]$ to $\delta_0$. Thanks to the presence of reversals, this functor satisfies the condition assumed on $R$ in Section 3, as proved in \cref{surj} below. Therefore, we may apply the results from Section 3 after checking that $\mathbf{D}_{r_\square}$ has a contractible nerve.
 
\begin{lemma}
\label[lemma]{surj}
  The functor $r_\square : \Delta_{a,\sharp} \to \square_\sharp^s$ induces a functor $$r_\square^\rightarrow : \Delta_{a,\sharp}^\rightarrow \to (\square_\sharp^s)^\rightarrow$$ between the arrow categories that is essentially surjective on objects.
\end{lemma}
 
 \begin{proof}
  First, observe that every map $f : I^n \to I^m$ in $\square_\sharp^s$ is essentially generated by the identity map $id_{I^0} : I^0 \to I^0$ and the face maps $\delta_i : I^0 \to I^1$ in the sense that $f$ factors as $w \circ f'$, where $w : I^m \to I^m$ is an automorphism, and $f' :  I^n \to I^m$ is a tensor product of $id_{I^0}$, $\delta_0$ and $\delta_1$.
 Recall that, by definition, $\square_\sharp^s$ is generated under the tensor product by $id_{I^0}$, $\delta_i$, the transposition $I^2 \to I^2$ and the reversal $I^1 \to I^1$. The observation can then be proved inductively on $m$. If $m=0$, then $n=0$ and $f$ is an identity arrow, so the result is obvious. Otherwise, either $n < m$ and $f$ factors through $I^{m-1}$, or $n = m$ and $f$ is an isomorphism. In the latter case, the result is obvious, so we are only left with the former: by assumption we have a diagram

\[\begin{tikzcd}[ampersand replacement=\&]
	{I^n} \&\& {I^{m-1}} \&\& {I^m} \\
	\&\& {I^{m-1}}
	\arrow["{f_0'}", from=1-1, to=1-3]
	\arrow["{f_0}"', from=1-1, to=2-3]
	\arrow["{f_1'}", from=1-3, to=1-5]
	\arrow["\simeq"{description}, from=1-3, to=2-3]
	\arrow["{f_1}"', from=2-3, to=1-5]
\end{tikzcd}\]
 where $f_0'$ is a tensor product of the form we want to obtain, by our inductive hypothesis, and where $f_1$ (and $f_1'$) must be a tensor of isomorphisms and exactly one face $\delta_i$. The face map can be composed with an identity arrow $I^0 \to I^0$ in the decomposition of $f_0'$ to provide an arrow $f' : I^n \to I^m$ of the form we want, and the remaining isomorphisms in the decomposition of $f_1'$ define, together with the identity arrow $I^1 \to I^1$ an automorphism $w$ of $I^m$ such that $w \circ f'$ equals $f$. This concludes the proof of our observation.
 
 Since $r_\square$ is monoidal, and leveraging our previous observation, it will be enough to check that the generators $id_{I^0}$, $\delta_0$ and $\delta_1$ are in the essential image of $r_\square^\rightarrow$ in order to conclude that every arrow of $\square_\sharp^s$ is. This is obvious for $id_{I^0}$ and $\delta_0$ (which are in the exact image of $r^\rightarrow$), and it is also the case of $\delta_1$, which can be written $\rho \circ \delta_0$, where $\rho : I^1 \to I^1$ is the reversal.  
 \end{proof}
 
 \begin{lemma}
 \label[lemma]{contractible_nerves}
  The category $\mathbf{D}_{r_\square}$ has a contractible nerve. Moreover, the full subcategories of  $\mathbf{D}_{r_\square}$ spanned by objects whose codomain is at most $n$ also have contractible nerves.
 \end{lemma}
 
 \begin{proof}
Consider the full subcategory $\iota_0 : \mathbf{D}_{r_\square,0} \subset \mathbf{D}_{r_\square}$ spanned by the arrows $\alpha : n \to m$ such that either $n=[-1]$ or there are no "free" automorphisms at the end. We claim that $\iota_0$ is homotopy final.

Consider a comma square
\[\begin{tikzcd}[ampersand replacement=\&]
	{\alpha \downarrow \iota_0} \&\& {\mathbf{D}_{\square^s_\sharp,0}} \\
	\\
	{*} \&\& {\mathbf{D}_{\square^s_\sharp}}
	\arrow[from=1-1, to=1-3]
	\arrow[from=1-1, to=3-1]
	\arrow["{\iota_0}", from=1-3, to=3-3]
	\arrow[between={0.3}{0.7}, Rightarrow, from=3-1, to=1-3]
	\arrow["\alpha"', from=3-1, to=3-3]
\end{tikzcd}\]
where $\alpha$ is of the form $\alpha_0 : n \to m$ followed optionally by a "free" automorphism $[w] : m \to m$. 

If $\alpha$ is in $\mathbf{D}_{r_\square,0}$, then $\alpha \downarrow \iota_0$ trivially has an initial object (given by the identity map $\alpha \to \alpha$), in particular $\alpha \downarrow \iota_0$ is contractible. Otherwise, $n$ is strictly greater than $-1$, and there is a "free" automorphism at the end. In this case,  $\alpha \downarrow \iota_0$ also has an initial object, namely the arrow $[-1] \to m \to m$ (ending with $[w]$), so it is also contractible.
This proves the claim.

Now consider the full subcategory $\iota_1 : \mathbf{D}_{r_\square,1} \to \mathbf{D}_{r_\square,0}$ spanned by the arrows containing no "free" automorphisms. This time, we claim that $\iota_1$ is homotopy initial.

Consider a comma square
\[\begin{tikzcd}[ampersand replacement=\&]
	{\iota_1 \downarrow \alpha} \&\& {\mathbf{D}_{r_\square,1}} \\
	\\
	{*} \&\& {\mathbf{D}_{r_\square,0}}
	\arrow[from=1-1, to=1-3]
	\arrow[from=1-1, to=3-1]
	\arrow[between={0.3}{0.7}, Rightarrow, nfold, from=1-3, to=3-1]
	\arrow["{\iota_1}", from=1-3, to=3-3]
	\arrow["\alpha"', from=3-1, to=3-3]
\end{tikzcd}\]
where, again, $\alpha$ is of the form $\alpha_0 : n \to m$ followed optionally by a "free" automorphism $[w] : m \to m$. 

As before, we may assume that $\alpha$ is not in $\mathbf{D}_{r_\square,1}$. This means that we can assume $\alpha$ to be of the form $[-1] \to [m] \to [m]$ for a "free" automorphism $[w] : [m] \to [m]$. This time, $\iota_1 \downarrow \alpha$ clearly has a terminal object, given by the unique factorization of $[-1] \to [m]$ through $\alpha$. Indeed, since $[-1]$ has no "free" automorphism, $m$ must be strictly greater than $[-1]$ and the free automorphism $[w]$ must appear in the bottom component of the factorization:
\[\begin{tikzcd}[ampersand replacement=\&]
	{[-1]} \&\& {[-1]} \\
	\&\& {[m]} \\
	{[m]} \&\& {[m]}
	\arrow["{\alpha_0}"', from=1-1, to=3-1]
	\arrow[from=1-3, to=1-1]
	\arrow["{\alpha_0}", from=1-3, to=2-3]
	\arrow["{[w]}", from=2-3, to=3-3]
	\arrow["{[w]}"', from=3-1, to=3-3]
\end{tikzcd}\]

Therefore, $\iota_1 \downarrow \alpha$ is contractible, which proves our second claim.
In particular, $\iota_0$ and $\iota_1$ induce weak equivalences of simplicial sets at the level of the nerves. To conclude that $\mathbf{D}_{r_\square}$ has a contractible nerve, it is thus enough to prove that the nerve of $\mathbf{D}_{r_\square,1}$ is contractible.

We define a functor $d : \mathbf{D}_{r_\square,1} \to \mathbf{D}_{r_\square,1}$ mapping an arrow $[m] \to [n]$ to $[-1] \to [n]$ (precomposing it with $[-1] \to [m]$). There is a natural transformation $id_{\mathbf{D}_{r_\square}} \to d$ whose components are given by the twisted squares below:
\[\begin{tikzcd}[ampersand replacement=\&]
	{[m]} \&\& {[-1]} \\
	\\
	{[n]} \&\& {[n]}
	\arrow[from=1-1, to=3-1]
	\arrow[from=1-3, to=1-1]
	\arrow[from=1-3, to=3-3]
	\arrow[from=3-1, to=3-3]
\end{tikzcd}\]
 
 Write $c : \mathbf{D}_{r_\square,1} \to \mathbf{D}_{r_\square,1}$ for the constant functor with value the arrow $[-1] \to [-1]$. There is a natural transformation 
 $c \to d$ defined by the squares of the following form: 

\[\begin{tikzcd}[ampersand replacement=\&]
	{[-1]} \&\& {[-1]} \\
	\\
	{[n]} \&\& {[-1]}
	\arrow[from=1-1, to=1-3]
	\arrow[from=1-1, to=3-1]
	\arrow[from=1-3, to=3-3]
	\arrow[from=3-3, to=3-1]
\end{tikzcd}\]
Overall, this provides a zig-zag between the identity functor on $\mathbf{D}_{r_\square,1}$ and a constant functor, so that the nerve of $\mathbf{D}_{r_\square,1}$ is contractible, concluding the proof that $\mathbf{D}_{r_\square}$ has a contractible nerve.

For the second assertion, it is enough to observe that the functors and natural transformations constructed above restrict suitably (i.e., they give rise to functors between the considered full subcategories and corresponding natural transformations).
\end{proof}

Write $\text{scFr}(\mathcal{T})$ for the category $\square^s_\sharp\text{-Fr}(\mathcal{T})$.
To sum up, we obtain the following result by specializing those of the previous sections:

\begin{theorem}
  The category $\text{scFr}(\mathcal{T})$ of semi-cubical frames in $\mathcal{T}$ enjoys the structure of a $\square^s_\sharp$-tribe, and the evaluation functor $$\mathbf{ev}_0 : \text{scFr}(\mathcal{T}) \to \mathcal{T}$$ establishes a DK-equivalence.
  
  Moreover, if $\mathcal{T}$ is a $\pi$-tribe, then so is $\text{scFr}(\mathcal{T})$.
 \end{theorem}
 
 \begin{proof}
  We know from \cref{surj} that the functor $r_\square : \Delta_{a,\sharp} \to \square_\sharp^s$ satisfies the expected condition required to apply the construction of \cref{def_dr}. The finiteness condition on $R := \square_\sharp^s$ is clearly satisfied, and the contractibility assumption that we use to establish the result for the evaluation functor is supplied by \cref{contractible_nerves}. 
  
  Therefore, the tribe structure on $\text{scFr}(\mathcal{T})$ follows from \cref{R_tribe}, and \cref{R_frames_th} further proves that it is a semi-cubical tribe structure. The assertion about the evaluation functor is an instance of \cref{ev0_funct}. Finally, the additional statement for $\pi$-tribes follows from \cref{pitribe_frames} (the category $\square^s_\sharp$ satisfies the required assumptions).
 \end{proof}

\newpage

\printbibliography

\end{document}